\documentclass[12pt]{amsart}
\usepackage[utf8]{inputenc}
\usepackage{graphicx}
\usepackage{epstopdf}
\usepackage{inputenc}
\usepackage{enumitem}
\usepackage{amsmath}
\usepackage{xcolor}
\usepackage{amssymb}
\usepackage{mathrsfs}

\usepackage[utf8]{inputenc}
\usepackage{graphicx}
\usepackage{epstopdf}
\usepackage{inputenc}
\usepackage{enumitem}
\usepackage{amsmath}
\usepackage{xcolor}
\usepackage{amssymb}
\usepackage{dsfont}
\usepackage[colorlinks]{hyperref}

\def\house#1{\setbox1=\hbox{$\,#1\,$}%
	\dimen1=\ht1 \advance\dimen1 by 2pt \dimen2=\dp1 \advance\dimen2 by 2pt
	\setbox1=\hbox{\vrule height\dimen1 depth\dimen2\box1\vrule}%
	\setbox1=\vbox{\hrule\box1}%
	\advance\dimen1 by .4pt \ht1=\dimen1
	\advance\dimen2 by .4pt \dp1=\dimen2 \box1\relax}

\newtheorem{theorem}{Theorem}
\newtheorem*{theorem*}{Theorem}

\newtheorem{lemma}{Lemma}

\newtheorem*{acknowledgements*}{Acknowledgements}
\newtheorem{remark}{Remark}

\newcommand{\ve}{\varepsilon}

\newcommand{\mcal}{\mathcal}

\newcommand{\gm}{\gamma}
\newcommand\blfootnote[1]{
	\begingroup
	\renewcommand\thefootnote{}\footnote{#1}
	\endgroup
}

\begin{document}
\title[Large values of quadratic Dirichlet $L$-functions]{Large values of quadratic Dirichlet $L$-functions over monic irreducible polynomial in $\mathbb{F}_q[t]$}
\author{Pranendu Darbar and Gopal Maiti }
\address[Pranendu Darbar]{Department of Mathematical Sciences\\
	Norwegian University of Science and Technology\\
	NO-7491, Trondheim, Norway}
\email{darbarpranendu100@gmail.com}

\address[Gopal Maiti]{Institut de Math\'ematiques de Marseille,
	Aix Marseille Universit\'e, U.M.R. 7373,
	Campus de Luminy, Case 907,
	13288 MARSEILLE Cedex 9, France}
\email{g.gopaltamluk@gmail.com}


 \begin{abstract}
 	We prove an $\Omega$-result for the quadratic Dirichlet $L$-function $|L(1/2, \chi_P)|$ over irreducible polynomials $P$ associated with the hyperelliptic curve of genus $g$ over a fixed finite field $\mathbb{F}_q$ in the large genus limit. In particular, we showed that for any $\epsilon\in (0, 1/2)$,
 	\[
 	\max_{\substack{P\in \mathcal{P}_{2g+1}}}|L(1/2, \chi_P)|\gg \exp\left(\left(\sqrt{\left(1/2-\epsilon\right)\ln q}+o(1)\right)\sqrt{\frac{g \ln_2 g}{\ln g}}\right),
 	\]
 	where $\mathcal{P}_{2g+1}$ is the set of all monic irreducible polynomial of degree $2g+1$. This matches with the order of magnitude of the Bondarenko--Seip bound. 
   \end{abstract}


\maketitle
  
\section{Introduction} 
 A fundamental problem in number theory is understanding the maximum size of $L$-functions on the critical line. The conjectured maximum sizes take shapes depending on the symmetry group attached to the family of $L$-functions, given in the work of Farmer, Gonek, and Hughes \cite{FGH}. The symmetry groups are classical compact groups: Unitary, Symplectic, and Orthogonal. One expects the analytic properties of the $L$-functions to be largely governed only by the symmetry type via Random Matrix Theory.
 
  In recent years, significant progress has been made in studying the maximum size of $L$-functions belonging to the Unitary family, for example, the Riemann zeta function and Dirichlet L-functions, described below. This article focuses on the large values of $L$-functions of Symplectic symmetry given in subsection 
\ref{Symplectic family}. Studying the maximum size of $L$-functions of Orthogonal symmetry listed in \cite{CF} is also interesting.  
 \subsection{Unitary family} Here and throughout this paper, for any $j\geq 2$, $\ln_j$ and $\log_j$ denote the $j$-th iterated natural logarithm and logarithm base $q$ respectively with $q\geq 3$. After the work of Balasubramanian and Ramachandra \cite{BR}, the main development of a method called the ``resonance method" due to Soundararajan \cite{Sound}, which connects the problem to a general maximization problem of a quadratic form. Recently, using long resonators, introduced in \cite{Ais}, Bondarenko and Seip \cite{BS} obtained an improved $\Omega$-result for the Riemann zeta function on the critical line. This long resonator technique has been applied to the other region of the critical strip to obtain many interesting results for the unitary family $\{\zeta(\sigma+it): \, t\geq 0\}$, and $\{L(\sigma, \chi):\, \chi \text{ is primitive Dirichlet character} \pmod q\}$ (see \cite{AMM}, \cite{AMMP}, \cite{BS3}). More precisely, for the zeta function, it is known \cite{BS} that 
 \[
 \max_{0\leq t\leq T}|\zeta(1/2+it)|\geq \exp\left((1+o(1))\sqrt{\frac{\ln T \ln_3 T}{\ln_2 T}}\right).
 \]
  Later the above result was improved and optimized by de la Bret\'{e}che and Tenenbaum \cite{BT} with a constant $\sqrt{2}$ instead of $1$ inside the exponent. Also, they showed that 
  \begin{align*}\label{large values of l functions}
  \max_{\substack{\chi\in X_q^+\\ \chi \neq \chi_0}} |L(1/2, \chi)|\geq \exp\left((1+o(1))\sqrt{\frac{\ln q \ln_3 q}{\ln_2 q}}\right),
  \end{align*} 
  where $X_q^+$ is the set of all even characters modulo $q$.
  The main ingredients of the long resonator method are the full orthogonality of the family, positivity and sparsity of the resonator. 
 In the last decades, parallel to the number fields, people are also interested to understand the behaviour of such functions over function fields. Recently, Doki\'{c} et al. \cite{DOKIC} studied the large values for the unitary family $L(\sigma, \chi)$, where $\chi$ varies over Dirichlet character modulo $Q\in \mathbb{F}_q[x]$ and $\sigma\in [1/2, 1]$. Their method is inspired by \cite{BT} via full GCD sums and obtained that as $d(Q)\to \infty$,
  \[
  \max_{\substack{\chi \pmod Q\\ \chi \text{ even }\\ \chi\neq \chi_0}}|L(1/2, \chi)|\geq \exp\left(\left(\frac{1}{2}\sqrt{\frac{\sqrt{q}+1}{\sqrt{q}-1}\ln q}+o(1)\right)\sqrt{\frac{\ln |Q| \ln_3 |Q|}{\ln_2 |Q|}}\right),
  \]
  where $Q$ is an irreducible polynomial and $|Q|=q^{\deg(Q)}$.
   \subsection{Symplectic family}\label{Symplectic family}
 An important $L$-function of symplectic type is the quadratic Dirichlet $L$-function $L(1/2, \chi_d)$ with fundamental discriminant $d$. Using resonator technique, Soundararajan \cite{Sound} showed that for sufficiently large $x$,
 \[
 \max_{x\leq |d|\leq 2x}L(1/2, \chi_d)\geq \exp\left(\left(\frac{1}{\sqrt{5}}+o(1)\right)\sqrt{\frac{\ln x}{\ln_2 x}}\right).
 \] 
  \noindent
 In \cite{AMMP}, C. Aistleitner et al. mentioned that it is challenging to implement a long resonator to detect large values of $L(\sigma, \chi_d)$  as $d$ varies over all fundamental discriminants in a range $|d|\leq x$, since the orthogonality between these characters are much more subtle and loss the positivity property. A similar level of difficulty applies for the subfamily $\{L(1/2, \chi_p): p\leq x\}$. 
 
  
\vspace{2mm}
\noindent
  \textbf{Large values of $L$-functions over irreducible polynomials:}\label{large values section}
 Our main goal is to study a function field analog of the above symplectic family.
 Before we enunciate the main theorem of this paper, we need to present
 some basic notations on function fields.
  Let $\mathbb{F}_{q}$ be a finite field of odd cardinality and $\mathbb{F}_{q}[t]$ be the polynomial ring over $\mathbb{F}_{q}$ in variable $t$. Let $P\in \mathbb{F}_q[t]$ be a monic irreducible polynomial. The quadratic character $\chi_P$ attached  to $P$ is defined using quadratic residue symbol for $\mathbb{F}_{q}[t]$ by $\chi_{P}(f)=\left(\frac{f}{P}\right)$ and the corresponding Dirichlet $L$-function is denoted by $L(s, \chi_P)$. It is often convenient to work with the equivalent $L$-function $\mathcal{L}(u, \chi_P)$ written in terms of the variable $u=q^{-s}$.
 Let $\mathcal{P}_n$
 be the family of all curves given in affine form by an equation $C_P: y^2=P(t)$, where $P$ be an irreducible polynomial of degree $n$. These curve are non-singular of genus $g$ given by \eqref{lambda} and \eqref{genus}.

  In \cite{Lumley} and \cite{Lumley2}, Lumley studied the distribution of $L(\sigma, \chi_D)$ with $D$ varies over $\mathcal{H}_{2g+1}$ and $\sigma\in (1/2, 1]$, where $\mathcal{H}_{2g+1}$ is a set of all monic square-free polynomial of odd degree $2g+1$. 
Later, the first author and Lumley \cite{DL} extend the distribution result to the half line by proving Selberg's central limit theorem under specific condition.

The hard part is to obtain large values of $L(1/2, \chi_D), D\in \mathcal{H}_{n}$ on the level of Bondarenko--Seip bound. The main object of this paper is to study large values of the subfamily $\{L(1/2, \chi_P)\}_{P\in \mathcal{P}_{2g+1}}$.
 We have the following theorem.    
 \begin{theorem}\label{main theorem}
 	Let $q\geq 3$ be fixed. As $g\to \infty$, for any $\epsilon\in (0, 1/2)$, we have
 	\[
 	\max_{P\in \mathcal{P}_{2g+1}}|L(1/2, \chi_P)|\gg \exp\left(\left( \sqrt{\left(1/2-\epsilon\right)\ln q}+o(1)\right)\sqrt{\frac{g \ln_2 g}{\ln g}}\right).
 	\]
 	\end{theorem}
 
 \begin{remark}
 In \cite{FGH}, Farmer, Gonek and Hughes conjectured that as $n\to \infty$,
 	\[
 	\max_{P\in \mathcal{P}_{n}}|L(1/2, \chi_P)|=\exp\left( \left(\sqrt{\ln q}+o(1)\right)\sqrt{n \ln n}\right),
 	\]
 	which is a particular case of Conjecture $C$. It seems very difficult to find large values of the subfamily $\{L(1/2, \chi_P)\}_{P\in \mathcal{P}_{2g+2}}$ as same strength of Theorem \ref{main theorem} due to lack of positivity to the extra $\lambda$-terms coming from the approximate functional equation (see Lemma \ref{Approximate Functional Equation}). 
 	\end{remark}

 \subsection{Essence of the paper} We adapt the resonance method in the function fields inspired by \cite{BS} to get sub-GCD sum instead of full GCD sum as in \cite{BT}. Also, one needs to analyze $\sum_{P\in \mathcal{P}_{2g+1}}L(1/2, \chi_P)|R(\chi_P)|^2$ instead of $\sum_{P\in \mathcal{P}_{2g+1}}L(1/2, \chi_P)^2|R(\chi_P)|^2$, where $R(\chi_P)$ is called resonator to be optimized. We have to go with the first moment to the $L$-function because the non-square contribution of the characters sum becomes large to the second moment.
We also restrict the sum over irreducible polynomials since the non-square contribution over irreducibles is significantly small and not affected by the long resonator, which is the main obstructer to the number fields case. Unlike Soundararajan's method, the Bondarenko-Seip approach achieves superior large values by incorporating the innovative use of a  ``long resonator".

 So, the main idea is to break the sum $\sum_{P\in \mathcal{P}_{2g+1}}L(1/2, \chi_P)|R(\chi_P)|^2$ into a square and non-square parts, coming from the character sum. Then we apply positivity to the term that comes from the square contribution, and the character sum over irreducible polynomials will control the non-square part, since $L(s, \chi_P)$ is a polynomial of degree $2g$ and hence roots are handled via Weil's theorem (Riemann hypothesis over algebraic curves). This phenomenon does not happen for the family $L(1/2, \chi_p)$, when $p$ varies over fundamental discriminants. Hence, it is a symplectic family where the $L$-function exhibits large values on the Bondarenko--Seip type bound level.
 
One can achieve the same result for the subfamily $\{L(1/2, \chi_p): p\leq x\}$ by adapting the above method under the conditional assumption of the Grand Riemann Hypothesis (GRH). The GRH plays a crucial role in mitigating the impact of large prime numbers on error terms arising from both square and non-square elements. Additionally, it is essential to adeptly manage the smooth function within the approximate functional equation. The use of function fields offers a distinct advantage, since $L(s, \chi_D)$ behaves as a polynomial and benefits from the application of Weil bound.    
 
 \section{Preliminaries and Lemmas} 
\subsection{Basic facts on $\mathbb{F}_q[t]$} We start by fixing a finite field $\mathbb{F}_{q}$ of odd cardinality $q=p^r$, $r\geq 1$ with an odd prime $p$. Let $\mathbb{A}:=\mathbb{F}_{q}[t]$ be the polynomial ring over $\mathbb{F}_{q}$.
For a polynomial $f$ in $\mathbb{F}_{q}[t]$, its degree is denoted by either $\deg(f)$ or $d(f)$. 

\vspace{1mm}
\noindent 
The set of all monic polynomials and monic irreducible polynomials of degree $n$ are denoted by $\mathcal{M}_{n,q}$ (or simply $\mathcal{M}_{n}$ as we fix $q$) and $\mathcal{P}_{n, q}$ (or simply $\mathcal{P}_{n}$) respectively. Let $\mathcal{M}=\cup_{n\geq 1} \mathcal{M}_{n}$ and $\mathcal{P}=\cup_{n\geq 1} \mathcal{P}_{n}$. We also denote the set of all monic polynomials of degree less or equal to $n$ by $\mathcal{M}_{\leq n}$.


If $f$ is a non-zero polynomial $\mathbb{F}_{q}[t]$ then the norm of $f$ is defined by $|f|=q^{d(f)}$. If $f=0$, we set $|f|=0$.
The Prime Polynomial Theorem (see \cite{ROS}, Theorem $2.2$) states that 
\begin{align}\label{prime poly th}
|\mathcal{P}_{n}|=\frac{q^{n}}{n} \,\,+\,\,O\Big(\frac{q^{\frac{n}{2}}}{n}\Big).
\end{align}
\noindent
The zeta function of $\mathbb{A}$, denoted by $\zeta_{\mathbb{A}}(s)$ and is defined by $$\zeta_{\mathbb{A}}(s):=\displaystyle\sum_{f\in \mathcal{M}}\frac{1}{|f|^{s}}=\displaystyle\prod_{P\in \mathcal{P}} \left( 1- |P|^{-s} \right)^{-1},\quad\quad  \Re(s)>1.$$ One can easily prove that $\zeta_{\mathbb{A}}(s)=\frac{1}{1-q^{1-s}}$, and this provides an analytic continuation of zeta function to the complex plane with a simple pole at $s=1$. Using the change of variable $u=q^{-s}$, 
\begin{align*}
\mathcal{Z}(u)=\sum_{f\in \mathcal{M}}u^{d(f)}=\frac{1}{1-qu}, \quad \text{ if } |u|<\frac{1}{q}.
\end{align*} 
\noindent
\subsection{Quadratic Dirichlet character and properties of their $L$-functions} For a monic irreducible polynomial $P$, the quadratic residue symbol $\left(\frac{f}{P} \right)$ is defined by 
\begin{align*}
\left( \frac{f}{P}\right)= 
\left\{
\begin{array}
[c]{ll}
1 & \text{if\, $f$ is a square\;} (\text{mod}\,\,  P),\,\, P\nmid f \\
-1 & \text{if\, $f$ is not a square\;} (\text{mod}\,\, P),\,\, P\nmid f\\
0 & \text{if \; } P\mid f.
\end{array}
\right.
\end{align*}
The quadratic Dirichlet character $\chi_{P}$ is defined by $\chi_{P}(f)= \left(\frac{f}{P} \right)$ and the associated $L$-function by
\begin{align*}
L(s,\chi_{P})=\sum_{f\in \mathcal{M}} \frac{\chi_{P}(f)}{|f|^{s}}=\prod_{Q\in \mathcal{P}}\left(1-\chi_{P}(Q)\,|Q|^{-s} \right)^{-1},\;\; \Re(s)>1.
\end{align*} 
Using the change of variable $u=q^{-s}$, we have  
\begin{align*}
\mathcal{L}(u,\chi_{P})=\sum_{f\in \mathcal{M} } \chi_{P}(f)\, u^{d(f)}=\prod_{Q\in\mathcal{P}}\left(1-\chi_{P}(Q)\,u^{d(Q)} \right)^{-1},\;\;\, |u|<\frac{1}{q} .
\end{align*}
\noindent
By [\cite{ROS}, Proposition $4.3$], we see that if $n\geq d(P)$ then $$\sum_{f\in \mathcal{M}_{n}}\chi_{P}(f)=0 .$$ It implies that $\mathcal{L}(u,\chi_{P})$ is a polynomial of degree at most $d(P)-1$. From \cite{Rud}, $\mathcal{L}(u, \chi_P)$ has a trivial zero at $u=1$ if and only if $d(P)$ is even. This allows us to define the completed $L$-function as
\[
L(s,\chi_{P})=\mathcal{L}(u, \chi_P)=(1-u)^{\lambda}\mathcal{L}^{*}(u, \chi_P)=(1-q^{-s})^{\lambda}{L}^{*}(s, \chi_P),
\]
where 
\begin{align}\label{lambda}
\lambda= 
\left\{
\begin{array}
[c]{ll}
1, & \;\text{if\, $d(P)$ even}, \\
0, &\; \text{if\, $d(P)$ odd},
\end{array}
\right.
\end{align}
and $\mathcal{L}^*(u, \chi_P)$ is a polynomial of degree 
\begin{align}\label{genus}
2g=d(P)-1-\lambda
\end{align}
satisfying the functional equation 
\[
\mathcal{L}^*(u, \chi_P)=(qu^2)^{g}\mathcal{L}^*\left(\frac{1}{qu}, \chi_P\right).
\] 
The Riemann hypothesis for curve over finite fields, established by Weil \cite{WEIL}, asserts that all the non-trivial zero of $\mathcal{L}^{*}(u, \chi_{P})$ are lie on the circle $|u|=q^{-1/2}$, i.e.,
\[ 
\mathcal{L}^{*}(u, \chi_{P})=\prod_{j=1}^{2g}\left(1-u \, \nu_{j}\, \right)\;\; \text{with}\;\, |\nu_{j}|=\sqrt{q}\;\, \text{for all}\,\, j.  
\]

\subsection{The resonator over function fields}
For a fixed $P\in \mathcal{P}_{2g+1}$, the resonator is a Dirichlet polynomial defined by
\[
R(\chi_P):= \sum_{h\in \mathcal{R}}\psi(h)\chi_P(h)
\]
with $|\mathcal{R}|\leq q^{\theta g}$ for some $\theta\leq 1/2$ such a way that the coefficients $\psi(h)$ are positive which resonates with $L(1/2, \chi_P)$, where $\mathcal{R}$ is constructed below. 
\begin{remark}
	We have simplified resonator compare to \cite{BS} and \cite{BT}. They needed to construct a subset $\mathcal{R}'\subset \mathcal{R}$ which control local factors of the form $\log (m/n)$ for the zeta and certain congruence classes for Dirichlet L-function in conductor aspect. In our case, $\mathcal{R}'= \mathcal{R}$, since for any $l\in \mathcal{R}'$,
	\[
	r(l)^2=\sum_{\substack{h\in \mathcal{R}\\ hl=\square}}\psi(h)^2=\psi(l)^2,
	\]  
	where $r(l)$ is a positive valued function on $\mathcal{R}'$.
	\end{remark}
  
\subsubsection*{Construction of $\mathcal{R}$}
Let $N$ be a large number to be chosen later.
 Let $a\in (1, \infty)$ and $\gamma\in (0,1)$ be fixed real numbers.
We consider $\mathbb{P}$ as the set of all irreducible polynomials $P$ such that $$\log{(q\ln N \ln_{2} N)}< d(P)\le \log{(q^{(\ln_{2} N)^{\gm}}\ln N \ln_{2} N)}.$$
Let us define a multiplicative function $\psi$ supported on the set of square-free polynomials such that for any  $P\in\mathbb{P}$,
\begin{align*}
\psi(P)=\sqrt{\frac{\ln N \ln_{2}N }{\ln_{3}N}}|P|^{-1/2}\left(d(P)-\log(\ln N \ln_{2}N)\right)^{-1}.
\end{align*}
 Let $\mathbb{P}_{k}$ be the set of all irreducible polynomials $P$ such that $$\log{(q^{k}\ln N \ln_{2} N)}< d(P)\le \log{(q^{k+1}\ln N \ln_{2} N)},$$
 where $k=1,\ldots,[(\ln_{2}N)^{\gm}]$. 
Fix $1<a<\frac{1}{\gm}$. Then, let $ \mathcal{R}_{k}$ be the set of polynomials that have at least $\frac{a\ln N}{k^{2}\ln_{3}N}$ irreducibles in $\mathbb{P}_{k}$, and let $ \mathcal{R}^{\prime}_{k}$ be the set of polynomials from $\mathcal{R}_{k}$ that have irreducibles only in $\mathbb{P}_{k}$. Finally set 
\[
\mathcal{R}:=\rm{supp}(\psi)\setminus\displaystyle \cup_{k=1}^{[(\ln_{2}N)^{\gm}]}\mathcal{R}_k.
\]
In other words, $\mathcal{R}$ is the set of square-free polynomials that have at most $\frac{a\ln N}{k^{2}\ln_{3}N}$ irreducibles in each group $\mathcal{R}_{k}$.
Note that, if $h\in\mathcal{R}$ then 
\begin{align*}
	d(h)\ll \ln N \ln_{2} N . 
\end{align*}	
\begin{lemma}\label{Le1} 
	For large enough $N$ and $1<a<\frac{1}{\gm}$, we have $|\mathcal{R}|\leq N$.
\end{lemma}

\begin{proof}
	From the construction of $\mathcal{R}$, we have
	\[
	|\mathcal{R}|\leq \prod_{k=1}^{\lfloor(\ln_{2}N)^{\gm}\rfloor}\sum_{j=0}^{\lfloor\frac{a\ln N}{k^{2}\ln_{3}N}\rfloor}\binom{|\mathbb{P}_k|}{j}.
		\]
		By the Prime Polynomial Theorem, we get
		$
		|\mathbb{P}_k|\leq \zeta_{\mathbb{A}}(2)q^{k+1} \ln N.
		$
		Therefore, following the arguments of Lemma 2 in \cite{BS2}, we have
		\[
		|\mathcal{R}|\leq \exp\left(\frac{a\ln N}{\ln_3 N}\sum_{k=1}^{\lfloor(\ln_{2}N)^{\gm}\rfloor}\frac1k\right)\leq \exp\left(a\gamma\ln N\right)\leq N.
		\]
		\end{proof}
 Let us define 
\begin{align*}
	\mcal{A}_{N}:=\frac{1}{\sum_{h\in\mcal{M}} \psi(h)^{2}}\sum_{f\in\mcal{M}}\frac{\psi{(f)}}{\sqrt{|f|}}\sum_{g|f}\psi(g)\sqrt{|g|}.
\end{align*}
Since $\psi$ is a multiplicative function, we obtain 
\begin{align}\label{eq2}
	\mcal{A}_{N}=\prod_{P\in\mathbb{P}}\frac{1+\psi(P)^2 +\psi(P)|P|^{-1/2}}{1+\psi(P)^2}.
\end{align}

\begin{lemma}\label{Le2}
As $N\to\infty$, we have
\[\mcal{A}_{N}\ge \exp\left((\gm + o(1))\sqrt{\frac{\ln N\ln_{3}N}{\ln_{2}N}}\right).
\]	
\end{lemma}

\begin{proof}
	Since $\psi(P)<(\ln_{3}N)^{-1/2}$ for all $P\in\mathbb{P}$, it follows from \eqref{eq2} that
	\[
	\mcal{A}_{N}=\exp\left((1+o(1))\sum_{P\in\mathbb{P}}\frac{\psi(P)}{\sqrt{|P|}}\right).
	\]
	Now, \begin{align*}
		\sum_{P\in\mathbb{P}}\frac{\psi(P)}{\sqrt{|P|}}
		&=\sqrt{\frac{\ln N \ln_{2}N }{\ln_{3}N}}\sum_{P\in\mathbb{P}}|P|^{-1} \left(d(P)-\log(\ln N \ln_{2}N)\right)^{-1}\\
		& =\sqrt{\frac{\ln N \ln_{2}N }{\ln_{3}N}}\sum_{\log{(q\ln N \ln_{2} N)}<m\le\log{(q^{(\ln_2 N)^{\gm}}\ln N \ln_{2} N)}} \frac{1}{m\left(m-\log(\ln N\ln_{2}N)\right)}\\
		& =\left(1+o(1)\right)\sqrt{\frac{\ln N \ln_{2}N }{\ln_{3}N}} \int_{\log{(q\ln N \ln_{2} N)}}^{\log{(q^{(\ln_2 N)^{\gm}}\ln N \ln_{2} N)}}\frac{dx}{x(x-\log(\ln N\ln_{2}N))} \\
		& \ge \left(\gm +o(1)\right)\sqrt{\frac{\ln N \ln_{3}N }{\ln_{2}N}},
	\end{align*}
	since $\int \frac{1}{x(x-c)}dx =\frac{\ln (x-c)-\ln x}{c}+\text{constant}$, and $\log A= \frac{\ln A}{\ln q}$.
	\end{proof}  
  
  \begin{lemma}\label{Le3}
	We have 
	\begin{align*}
	\frac{1}{\sum_{h\in\mcal{M}} \psi(h)^{2}}\sum_{f\in\mcal{M};\,  f\notin\mathcal{R}}\frac{\psi{(f)}}{\sqrt{|f|}}\sum_{g|f}\psi(g)\sqrt{|g|}=o(\mcal{A}_{N}), \;\; N\to\infty.
	\end{align*}	
		
\end{lemma}  
 
 \begin{proof}
 	Note that, \begin{align}\label{eq4}
 			\frac{1}{\mcal{A}_{N}\sum_{h\in\mcal{M}} \psi(h)^{2}}\sum_{f\in\mcal{M}; \, f\notin \mathcal{R}}\frac{\psi{(f)}}{\sqrt{|f|}}\sum_{g|f}\psi(g)\sqrt{|g|} \nonumber \hspace{3cm}\\
 			\le\frac{1}{\mcal{A}_{N}\sum_{h\in\mcal{M}} \psi(h)^{2}}\sum_{k=1}^{[(\ln_{2}N)^{\gm}]}\sum_{f\in\mathcal{R}_{k}}\frac{\psi{(f)}}{\sqrt{|f|}}\sum_{g|f}\psi(g)\sqrt{|g|}.
 	\end{align}
 For each $k=1,\ldots,[(\ln_{2}N)^{\gm}]$, we have 
 \begin{align}\label{eq5}
 	\frac{1}{\mcal{A}_{N}\sum_{h\in\mcal{M}} \psi(h)^{2}}\sum_{f\in\mathcal{R}_{k}}\frac{\psi{(f)}}{\sqrt{|f|}}\sum_{g|f}\psi(g)\sqrt{|g|} \nonumber\hspace{4.5cm}\\
 	=\frac{1}{\prod_{P\in\mathbb{P}_{k}}\left(1+\psi(P)^2+\psi(P)|P|^{-1/2}\right)}\sum_{f\in\mathcal{R}_{k}^{\prime}}\frac{\psi{(f)}}{\sqrt{|f|}}\sum_{g|f}\psi(g)\sqrt{|g|} \nonumber\\
 	\le\frac{1}{\prod_{P\in\mathbb{P}_{k}}\left(1+\psi(P)^2\right)}\sum_{f\in\mathcal{R}_{k}^{\prime}}\psi(f)^{2}\prod_{P\in\mathbb{P}_{k}}\left(1+\frac{1}{\psi(P)\sqrt{|P|}}\right). \hspace{5mm}
 \end{align}
First we make the following computation:
\begin{align*}
\prod_{P\in\mathbb{P}_{k}} \left(1+\frac{1}{\psi(P)\sqrt{|P|}}\right)= \hspace{11cm}\\  \prod_{\substack{
	\log{(q^{k}\ln N \ln_{2} N)}< d(P)\le \log{(q^{k+1}\ln N \ln_{2} N)} }}
   \left(1+ (d(P)-\log(\ln N\ln_{2}N ) )\sqrt{\frac{\ln_{3}N}{\ln N\ln_{2}N}}\right). 
\end{align*}
Since $(1+x)^{c}\le e^{cx}$, this implies
\begin{align}\label{eq6}
\prod_{P\in\mathbb{P}_{k}} \left(1+\frac{1}{\psi(P)\sqrt{|P|}}\right)   & \le \left(1+ (k+1)\sqrt{\frac{\ln_{3}N}{\ln N\ln_{2}N}}\right)^{q^{k+1} \ln N} \nonumber\\
 &  \le \exp\left((k+1)q^{k+1} \sqrt{\frac{\ln N \ln_{3}N}{\ln_{2}N}} \right)=\exp\left(o\left(\frac{\ln N}{\ln_{3}N}\right)\frac{1}{k^2}\right),
\end{align}
as $k\le (\ln_{2}N)^{\gm}$.
Since every $f\in\mathcal{R}_{k}^{\prime}$ has at least $\frac{a\ln N}{k^2 \ln_{3}N}$ irreducibles and $\psi(f)$ is multiplicative function, it therefore follows that 
\[
\sum_{f\in\mathcal{R}_{k}^{\prime}}\psi(f)^{2}\le b^{-a\frac{\ln N}{k^2 \ln_{3}N}}\prod_{P\in\mathbb{P}_{k}}\left( 1+ b\psi(P)^2 \right),
\] whence $b>1$. Hence,
 \begin{align}\label{eq7}
 	\frac{1}{\prod_{P\in\mathbb{P}_{k}}\left(1+\psi(P)^2\right)}\sum_{f\in\mathcal{R}_{k}^{\prime}}\psi(f)^{2}\le b^{-a\frac{\ln N}{k^2 \ln_{3}N}}\exp\left((b-1)\sum_{Q\in\mathbb{P}_{k}}\psi(Q)^2 \right).
 \end{align}
Finally, 
\begin{align*}
\sum_{P\in\mathbb{P}_{k}}\psi(P)^2 &=\frac{\ln N \ln_{2}N}{\ln_{3}N } \sum_{P\in\mathbb{P}_{k}} |P|^{-1}\left(d(P)-\log(\ln N \ln_{2}N) \right)^{-2}\\
&\le\left(1+o(1)\right)\frac{\ln N \ln_{2}N}{\ln_{3}N} \int_{\log{(q^{k}\ln N \ln_{2} N)}}^{\log{(q^{k+1}\ln N \ln_{2} N)}} \frac{1}{k^2 x}dx  \\
&\le \left(1+o(1)\right)\frac{\ln N}{k^2 \ln_{3}N}.
\end{align*}
 Combining the last inequality with \eqref{eq6} and \eqref{eq7}, we get that \eqref{eq5} is at most \[ \exp\left((b -1-a\ln b +o(1))\frac{\ln N}{k^2 \ln_{3}N}\right) .\]
Note that $a>1$. Choosing $b$ sufficiently close to $1$, we obtain $(b -1)-a\ln b <0$, and completes the proof.
\end{proof} 

\begin{lemma}\label{Le5}
	Let $\ve$ be a positive real number. Then as $N\to\infty$, 
	\[ 
	\frac{1}{\sum_{h\in\mcal{M}} \psi(h)^{2}} \sum_{f\in\mathcal{R}}\frac{\psi{(f)}}{\sqrt{|f|}}\sum_{\substack{g|f\\ d(g)\le d(f)-\ve\log N}}\psi(g)\sqrt{|g|}=o(\mcal{A}_{N}),
	\]
	where the implicit constant only depends on $\ve$.
\end{lemma}

\begin{proof}
	We have \begin{align*}
	\sum_{f\in\mathcal{R}}\frac{\psi{(f)}}{\sqrt{|f|}}\sum_{\substack{g|f\\ d(g)\le d(f)-\ve\log N}}\psi(g)\sqrt{|g|}=\sum_{f\in\mathcal{R}}\psi{(f)}^2 \sum_{\substack{l|f\\ d(l)\ge \ve\log N}} \frac{1}{\psi{(l)}\sqrt{|l|}}.
	\end{align*}
	It is therefore enough to show that, for each $f\in\mathcal{R}$,   
	\[
	\sum_{\substack{l|f\\ d(l)\ge \ve\log N}} \frac{1}{\psi{(l)}\sqrt{|l|}}=o(1),\;\; N\to\infty.
	\]
	Finally, 
	\[
	\sum_{\substack{l|f\\ d(l)\ge \ve\log N}} \frac{1}{\psi{(l)}\sqrt{|l|}}\le N^{-\ve/4}\sum_{l|f} \frac{1}{\psi{(l)}|l|^{1/4}}= N^{-\ve/4}\prod_{P|f}\left(1+\frac{1}{\psi{(P)}|P|^{1/4}} \right)=o(1),
	\]
 since $\frac{1}{\psi{(P)}|P|^{1/4}}=o(1)$ uniformly for all $P\in\mathbb{P}$, and the polynomial $f$ has at most $\frac{a\ln N}{\ln_{3}N}$ irreducible polynomials.
	
\end{proof}

\begin{lemma}\label{Le4} Let $l\in \mathcal{M}$. Then we have 
\[
\sum_{P\in\mcal{P}_{n}}\chi_{P}(l^{2})=\frac{q^{n}}{n}+O\left(\frac{q^{n/2}}{n}\right)+O(\log \log d(l)).
\]
\begin{proof}
	It is a direct application of Prime Polynomial Theorem \eqref{prime poly th}.
	\end{proof}
\end{lemma}
We need the following crucial lemma that estimates a certain character sum over irreducible using the known Riemann hypothesis for curves.
\begin{lemma}[Weil bound]\label{lem5}
Let $f\in \mathcal{M}$ and $f$ is not a perfect square, then
\[
\Big|\sum_{P\in\mcal{P}_{n}}\chi_{P}(f)\Big|\ll \frac{d(f)}{n} q^{n/2}.
\]	
\end{lemma}	
\begin{proof}
	See equation $(2.5)$ of \cite{Rud}.
	\end{proof}
\begin{lemma}[Approximate Functional Equation]\label{Approximate Functional Equation}
Let $\chi_P$ be a quadratic Dirichlet character, where $P\in \mathcal{P}_n$. Then 
\begin{align*}
	L(1/2,\chi_{P})&=\sum_{f\in \mathcal{M}_{\leq g}}\frac{\chi_{P}(f)}{|f|^{1/2}}\; +  \sum_{f\in \mathcal{M}_{\leq g-1}}\frac{\chi_{P}(f)}{|f|^{1/2}}\\
	&-\lambda q^{-(g+1)/2}\sum_{f\in \mathcal{M}_{\leq g}}\chi_P(f)- \lambda q^{-g/2} \sum_{f\in \mathcal{M}_{\leq g-1}}\chi_P(f),
\end{align*} 	
where $n,\, g,\, \lambda$ are related by equations \eqref{lambda} and \eqref{genus}.
\end{lemma}
\begin{proof} 
	The case $s = \frac12$ is proved in \cite{AK} for $D \in\mathcal{H}_{2g+1}$ and \cite{J} for $D \in\mathcal{H}_{2g+1}$. Exactly the same way, one can prove it for $P \in\mathcal{P}_{n}$.
\end{proof}

\section{Proof of the Theorem}
Recall that $\gamma\in (0, 1)$ and $\mathcal{R}$ are set of square-free polynomials with
 $|\mathcal{R}|\le N$ from Lemma \ref{Le1}. Later, we will choose the size of $N$ appropriately.
 
\noindent
We define
\begin{align*} \mcal{S}_{1}:=\sum_{P\in\mcal{P}_{2g+1}}|R(\chi_P)|^2,
\quad \text{ and } \quad
 \mcal{S}_{2}:=\sum_{P\in\mcal{P}_{2g+1}}L\left(\frac{1}{2},\chi_{P}\right)|R(\chi_P)|^2 ,
\end{align*}
where $  
R(\chi_P):=\sum_{h\in\mathcal{R}}\psi(h)\chi_{P}(h)
$ and $\psi$ is a non-negative function on $\mathcal{M}$.
So,
\[ \max_{P\in\mcal{P}_{2g+1}} \big|L(1/2,\chi_{P})\big|\ge \frac{\mcal{S}_2}{\mcal{S}_1}. 
\]

\noindent
First we compute a lower bound of $\mcal{S}_{2}$. Using Lemma \ref{Approximate Functional Equation}, we have
\begin{align*}
	\mcal{S}_{2}=&\sum_{P\in\mcal{P}_{2g+1}}L\left(\tfrac{1}{2},\chi_{P}\right)|R(\chi_P)|^2 \\
	=&\sum_{h_1 ,h_2 \in \mcal{R}} \psi(h_1) \psi(h_2)\sum_{f\in\mcal{M}_{\le g}} \frac{1}{\sqrt{|f|}}\sum_{P\in\mcal{P}_{2g+1}}\chi_{P}(fh_1 h_2) \\
	& + \sum_{h_1 ,h_2 \in \mcal{R}}\psi(h_1) \psi(h_2)\sum_{f\in\mcal{M}_{\le g-1}} \frac{1}{\sqrt{|f|}}\sum_{P\in\mcal{P}_{2g+1}}\chi_{P}(fh_1 h_2)\\
     :=& \mcal{S}_{21} + \mcal{S}_{22}.
\end{align*}
Separate the innermost into square and non-square parts as $fh_1h_2=\square$ and $fh_1h_2\neq \square$. So, using Lemma \ref{Le4} and Lemma \ref{lem5}, we obtain	
\begin{align*}	
    \mcal{S}_{21}&=|\mcal{P}_{2g+1}| \sum_{h_1 ,h_2 \in \mcal{R}}\psi(h_1) \psi(h_2)\sum_{\substack{f\in\mcal{M}_{\le g}\\ h_1 h_2 f=\square}} \frac{1}{\sqrt{|f|}}  +O\Bigg(\frac{q^{g}}{g}\sum_{\substack{f\in\mcal{M}_{\le g}\\h_1 ,h_2 \in \mcal{R}\\ h_1 h_2 f\neq\square}} \frac{d(fh_1 h_2)\,\psi(h_1) \psi(h_2)}{\sqrt{|f|}}\Bigg),\\
	\mcal{S}_{22}&=|\mcal{P}_{2g+1}| \sum_{h_1 ,h_2 \in \mcal{R}}\psi(h_1) \psi(h_2)\sum_{\substack{f\in\mcal{M}_{\le g-1}\\ h_1 h_2 f=\square}} \frac{1}{\sqrt{|f|}} +O\Bigg(\frac{q^{g}}{g}  \sum_{\substack{f\in\mcal{M}_{\le g-1}\\h_1 ,h_2 \in \mcal{R}\\ h_1 h_2 f\neq\square}} \frac{d(fh_1 h_2)\,\psi(h_1) \psi(h_2)}{\sqrt{|f|}}\Bigg).
	\end{align*}
Since the coefficients $\psi(h)$ are non-negative, picking the term only when $fh_1=h_2$, we end up with
\begin{align*}
	\sum_{h_1 ,h_2 \in \mcal{R}}\psi(h_1) \psi(h_2) \sum_{\substack{f\in\mcal{M}_{\le g}\\ h_1 h_2 f=\square}} \frac{1}{\sqrt{|f|}} \ge \sum_{h_2 \in\mcal{R}} \frac{\psi(h_2)}{\sqrt{|h_2|}} \sum_{\substack{h_1 |h_2\\ d(h_1)\ge d(h_2)-\ve g}} \psi(h_1)\sqrt{|h_1|}.
\end{align*}
An important observation that $d(fh_1h_2)\ll g+\ln N \ln_{2} N$, which helps to control the non-square part due to Weil's bound. Therefore
\begin{align*}
	 \mcal{S}_{21}&\ge|\mcal{P}_{2g+_1}|  \sum_{h_2 \in\mcal{R}} \frac{\psi(h_2)}{\sqrt{|h_2|}} \sum_{\substack{h_1 |h_2\\ d(h_1)\ge d(h_2)-\ve g}} \psi(h_1)\sqrt{|h_1|} + O\left(q^{\frac{3g}{2}} N(g+\ln N \ln_{2} N) \left(\sum_{\substack{h\in \mcal{R}}}\psi(h)^{2}\right) \right)\\
	 &\ge |\mcal{P}_{2g+1}| \exp\left((\gm+ o(1))\sqrt{\frac{\ln N\ln_{3}N}{\ln_{2}N}}\right) \left(1+ O\left(\frac{N(g+\ln N \ln_{2} N) }{q^{\frac{g}{2}}}\right) \right)\left(\sum_{\substack{h\in \mcal{R}}}\psi(h)^{2}\right).
\end{align*}	 
Take $\ve= \frac{1}{2} -\ve^{\prime}$ 
 and $N\asymp q^{g(\frac{1}{2}-\ve^{\prime})}$, where $0<\ve^{\prime}<1/2$ is arbitrary small. We use Lemma \ref{Le2}, Lemma \ref{Le3}, and Lemma \ref{Le5} to get
\begin{align*}	 
	\mcal{S}_{21}  \gg \left(1+ o(1) \right) |\mcal{P}_{2g+1}| \exp\left((\gm + o(1))\sqrt{\frac{\ln N\ln_{3}N}{\ln_{2}N}}\right) \left(\sum_{\substack{h\in \mcal{R}}}\psi(h)^{2}\right).
\end{align*}
Similarly, we have
 \begin{align*}	 
 	\mcal{S}_{22}  \gg \left(1+ o(1) \right) |\mcal{P}_{2g+1}| \exp\left((\gm+ o(1))\sqrt{\frac{\ln N\ln_{3}N}{\ln_{2}N}}\right) \left(\sum_{\substack{h\in \mcal{R}}}\psi(h)^{2}\right).
 \end{align*}
Therefore, 
\begin{align*}	 
	\mcal{S}_{2}  \gg \left(1+ o(1) \right) |\mcal{P}_{2g+1}| \exp\left((\gm+ o(1))\sqrt{\frac{\ln N\ln_{3}N}{\ln_{2}N}}\right) \left(\sum_{\substack{h\in \mcal{R}}}\psi(h)^{2}\right).
\end{align*}

Now we estimate an upper bound of $\mcal{S}_1$. Following the above arguments and applying Cauchy-Schwarz inequality, we have
\begin{align*}
	\mcal{S}_{1}&=\sum_{P\in\mcal{P}_{2g+1}}|R(\chi_P)|^2=\sum_{h_1 ,h_2 \in \mcal{R}}\psi(h_1) \psi(h_2)\sum_{P\in\mcal{P}_{2g+1}}\chi_{P}(h_1 h_2)\\
	            &=|\mcal{P}_{2g+1}|	\sum_{\substack{h_1 ,h_2 \in \mcal{R} \\ h_1 h_2 =\square}}\psi(h_1) \psi(h_2) + O\left(\frac{q^{g}}{g}\sum_{\substack{h_1 ,h_2 \in \mcal{R} \\ h_1 h_2 \neq\square}}d(h_1 h_2) \psi(h_1) \psi(h_2) \right)\\
	            &\le|\mcal{P}_{2g+1}|	\sum_{\substack{h \in \mcal{R}}}\psi(h)^{2} + O\left(\frac{q^{g}}{g}(g+\ln N \ln_{2} N) \, |\mcal{R}|\sum_{\substack{h\in \mcal{R}}}\psi(h)^{2} \right)\\
	            &\le \left(1+O\left(\frac{N (g+\ln N \ln_{2} N) }{q^{g}}\right)\right)|\mcal{P}_{2g+1}|\left(\sum_{\substack{h\in \mcal{R}}}\psi(h)^{2}\right).
\end{align*}	
Therefore, the above choice of $N\asymp q^{g\left(\frac{1}{2}-\ve^{\prime}\right)}$ gives
\begin{align*}            
  \mcal{S}_{1} \le \left(1+o(1)\right)|\mcal{P}_{2g+1}|
	            \left(\sum_{\substack{h\in \mcal{R}}}\psi(h)^{2}\right).
\end{align*}

Hence for sufficiently large $g$ and for arbitrary small $\ve^{\prime}>0$, we get

 \begin{align*}
	\max_{P\in\mcal{P}_{2g+1}} \big|L(1/2,\chi_{P})\big|&\ge \frac{\mcal{S}_2}{\mcal{S}_1}\gg \exp\left((\gm+ o(1))\sqrt{\frac{\ln N\ln_{3}N}{\ln_{2}N}}\right)\\
	&=\exp\left((\gm+ o(1))\,
\sqrt{(\tfrac{1}{2}-\ve^{\prime})\ln q} \sqrt{\frac{g\ln_{2}g}{\ln g}}\right),
\end{align*}
which finishes the proof.\\

\textbf{Acknowledgment:} 
The authors would like to thank Kristian Seip for reading of the first version of this manuscript, valuable comments, and suggestions.
 The first author is funded by Grant 275113 of the Research Council of Norway through the Alain Bensoussan Fellowship Programme of the European Research
Consortium for Informatics and Mathematics. The second author is funded by the joint FWF-ANR project Arithrand: FWF: I 4945-N and ANR-20-CE91-0006.

  \end{document}